\documentclass[11pt,a4paper]{article}
\title{\textsc{Local optimality of\\ Zaks-Perles-Wills simplices}}

\author{Gennadiy Averkov\footnote{Faculty of Mathematics, Otto-von-Guericke-Universit\"at Magdeburg, Universit\"atsplatz 2, 39106 Magdeburg, Germany. Email: averkov@ovgu.de}}

\usepackage{amsthm}
\usepackage{amsmath}
\usepackage{amssymb}
\usepackage{enumerate}
\usepackage{graphicx}
\usepackage{tikz}
\newcommand{\rmcmd}[1]{\mathop{\mathrm{#1}}\nolimits}
\newcommand{\aff}{\rmcmd{aff}}
\newcommand{\lin}{\rmcmd{lin}}

\newcommand{\R}{\mathbb{R}}

\newcommand{\Z}{\mathbb{Z}}

\newcommand{\cP}{\mathcal{P}}

\newcommand{\cS}{\mathcal{S}}

\newcommand{\setcond}[2]{\left\{#1 \, : \, #2 \right\}}

\newcommand{\vol}{\rmcmd{vol}}

\newcommand{\conv}{\operatorname{conv}}
\newcommand{\intr}[1]{\operatorname{int}(#1)}
\newcommand{\relintr}[1]{\operatorname{relint}(#1)}

\newcommand{\slice}{\operatorname{slice}}

\newtheorem{nn}{}
\newtheorem{theorem}[nn]{Theorem}

\newtheorem{conjecture}[nn]{Conjecture}

\newtheorem*{acknowledgments*}{Acknowledgments}

\begin{document}

\maketitle

\begin{abstract}	
	In 1982, Zaks, Perles and Wills discovered a $d$-dimensional lattice simplex $S_{d,k}$ with $k$ interior lattice points, whose volume is linear in $k$ and doubly exponential in the dimension $d$. It is conjectured that, for all $d \ge 3$ and $k \ge 1$, the simplex $S_{d,k}$ is a volume maximizer in the family $\cP^d(k)$ of all $d$-dimensional lattice polytopes with $k$ interior lattice points. To obtain a partial confirmation of this conjecture, one can try to verify it for a subfamily of $\cP^d(k)$ that naturally contains $S_{d,k}$ as one of the members. Currently, one does not even know whether $S_{d,k}$ is optimal within the family $\cS^d(k)$ of all $d$-dimensional lattice simplices with $k$ interior lattice points. In view of this, it makes sense to look at even narrower families, for example, some subfamilies of $\cS^d(k)$. The simplex $S_{d,k}$ of Zaks, Perles and Wills has a facet with only one lattice point in the relative interior. We show that $S_{d,k}$ is a volume maximizer in the family of simplices $S \in \cS^d(k)$ that have a facet with one lattice point in its relative interior. We also show that, in the above family, the volume maximizer is unique up to unimodular transformations.
\end{abstract}

\section{Introduction}

	Let $d$ be a positive integer and $k$ a non-negative integer,
	let $o$ denote the origin and $e_1,\ldots,e_d$ the standard basis of the space $\R^d$. 
	A \emph{lattice polytope} in $\R^d$ is a polytope whose all vertices belong to the integer lattice $\Z^d$; see also \cite{Barvinok1997,GritzmannWills1993,Gruber2007,Barvinok2008} for background information. By $\vol$ we denote the $d$-dimensional volume (i.e., the Lebesgue measure) in the space $\R^d$, scaled in the usual way so that the unit cube $[0,1]^d$ has volume one. We call a map $\phi : \R^d \to \R^d$ a \emph{unimodular transformation} if $\phi$ is an affine transformation satisfying $\phi(\Z^d) = \Z^d$. We study the relationship between the volume and the number of interior lattice points for lattice polytopes. Both these functionals are invariant under unimodular transformations.

	Let $\cP^d(k)$ denote the family of all $d$-dimensional lattice polytopes in $\R^d$ with $k$ interior lattice points and $\cS^d(k)$ the family of all simplices belonging to $\cP^d(k)$. For $d=1$, up to unimodular transformations, the segment $[0,k+1]$ is the only member of $\cP^d(k)$ and $\cS^d(k)$, but for larger dimensions $d\ge 2$, $\cP^d(k)$ and $\cS^d(k)$ contain many different polytopes. 
	It is known that, for every $k \ge 1$, the volume of polytopes in $\cP^d(k)$ is bounded; see \cite{Hensley83}. The assumption $k \ge 1$ is necessary for boundedness, as for $d \ge 2$ and $k=0$, the volume of polytopes in $\cP^d(k)$ and $\cS^d(k)$ is unbounded; for example, the horizontal slab $\R^{d-1} \times \R$ contains lattice polytopes of arbitrarily large volume. In the last four decades, many researches tried to determine possibly tight volume bounds for the families $\cP^d(k)$ and $\cS^d(k)$ and their subfamilies; see \cite{Scott76,Hensley83,Pikhurko01,Conrads02,LagariasZiegler91,Nill07,Kasprzyk09,Averkov12,AverkovKruemplmannNill15,arXiv:BallettiKasrpzyk16,arXiv:BallettiKasprzykNill16,AKN2017}. Despite the constant progress, up to now, sharp volume bounds in $\cP^d(k)$ and $\cS^d(k)$ are known in just a few special cases.

	Volume bounds for $\cP^d(k)$ and its subfamilies have various applications. Such bounds were used in  \cite{AverkovWagnerWeismantel11,ACDDF13,CDDFG15} in the context of integer optimization. In the theory of toric varieties, volume and the number of interior lattice points of a polytopes are endowed with an an algebraico-geometric meaning; see \cite{Fulton1993,CoxLittleSchenck2011}. The number of interior lattice points and the volume are a part of the information provided by the Ehrhart polynomial of a lattice polytope; see \cite[\S19.1]{Gruber2007}. Hence, for understanding the structure of Ehrhart polynomials of general lattice polytopes, it is also necessary to understand the relationship between the volume and the number of interior lattice points. 

	In 1982, Zaks, Perles and Wills \cite{ZaksPerlesWills82} discovered the simplex 
	\[
		S_{d,k} := \conv \bigl(o,s_1 e_1, \ldots, s_{d-1} e_{d-1}, (k+1) (s_d-1) e_d \bigr),
	\] 
	derived from the so-called \emph{Sylvester sequence}, which is defined recursively by 
	\begin{equation}
		s_i:= \begin{cases}
			2 & \text{if} \ i=1,
			\\ 1 + s_1 \cdots s_{i-1} & \text{if} \ i \ge 2.
		\end{cases}
	\end{equation}
	
	We call $S_{d,k}$ the \emph{Zaks-Perles-Wills simplex}. The original definition from \cite{ZaksPerlesWills82} is restricted to the case $k \ge 1$, but we also include the case $k=0$, which is also interesting; see \cite[Remark~3.10]{AverkovWagnerWeismantel11} and Conjecture~\ref{conj:k=0} at the end of this paper. The simplex $S_{d,k}$ has $k$ interior lattice points and its volume 
	\[
		\vol(S_{d,k}) = \frac{1}{d!} (k+1) (s_d-1)^2
	\]
	is doubly exponential in the dimension. Recently, the following conjecture about the maximum volume in $\cP^d(k)$ was formulated.

\begin{conjecture}[Balletti \& Kasprzyk \cite{arXiv:BallettiKasrpzyk16}]
	\label{conj} Let $d \ge 3$ and $k \ge 1$. Then $S_{d,k}$ is a volume maximizer in $\cP^d(k)$. Furthermore, with the exception of the case $d=3, \, k=1$, the volume maximizer in $\cP^d(k)$ is unique up to unimodular transformations.
\end{conjecture}

Balletti and Kasprzyk \cite{arXiv:BallettiKasrpzyk16} point out that hints to Conjecture~\ref{conj} can also be found in older literature  \cite{Hensley83,ZaksPerlesWills82,LagariasZiegler91}. We give a short summary of the current knowledge of volume bounds for $\cP^d(k)$ and $\cS^d(k)$, with $k \ge 1$. In 1976, Scott \cite{Scott76}	determined the sharp volume bound in $\cP^2(k)$ and $\cS^2(k)$. Volume maximizers in $\cP^2(k)$ and $\cS^2(k)$ can deduced from refinements of Scott's result \cite{Scott76} presented in \cite{HaaseSchicho09}; see also Figure~\ref{fig:d=3:k=2} for an example in the case $k=2$.  Conjecture~\ref{conj} was verified by complete enumeration of $\cP^d(k)$ in the cases $d=3,\,k=1$ and $d=3,\, k=2$ in \cite{Kasprzyk09}  and \cite{arXiv:BallettiKasrpzyk16}, respectively. The currently best upper volume bound $(d(2d+1)(s_{2d+1}-1))^d k$ for the whole family $\cP^d(k)$ is much larger than $\vol(S_{d,k})$; see \cite{AKN2017}. In contrast to this, for the family $\cS^d(k)$, the best currently known bound $\frac{1}{d!} (d+1) (s_d-1)^2 k$ on the volume  differs only by a linear factor in $d$ from the conjectured bound $\vol(S_{d,k})$; see \cite{AKN2017}. It known that $S_{d,k}$ is a unique volume maximizer in $\cS^d(k)$ for every $d \ge 4$ and $k=1$; see \cite{AverkovKruemplmannNill15}.

\begin{figure}
	\begin{center}
		\begin{tikzpicture}[scale=0.3]
		\filldraw[fill=green!30!white] (0,0) -- (0,6) -- (2,0) -- cycle;
		\filldraw[fill=green!30!white] (4,0) -- (4,5) -- (4+2,1) -- (4+2,0) -- cycle;
		\filldraw[fill=green!30!white] (8,0) -- (8,4) -- (8+2,2) -- (8+2,0) -- cycle;
		\filldraw[fill=green!30!white] (12,0) -- (12,3) -- (12+2,3) -- (12+2,0) -- cycle;
		\foreach \x in {-1,...,15}
		\foreach \y in {-1,...,7}
		{
			\fill (\x,\y) circle (0.1);
		}
		\end{tikzpicture}
		\caption{\label{fig:d=3:k=2}. All volume maximizers in $\cP^2(2)$}
	\end{center}
\end{figure}
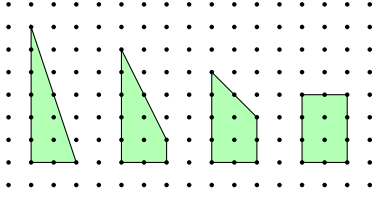

For $l \ge 0$, we introduce the subfamily $\cS^d(k,l)$ of $\cS^d(k)$ consisting of simplices that have a facet with exactly $l$ lattice points in its relative interior.
Our main result verifies the optimality of $S_{d,k}$ in the family $\cS^d(k,1)$. This family naturally includes $S_{d,k}$, because the point $(1,\ldots,1,0)$ is a unique lattice point in the relative interior of the facet $\conv(o,s_1 e_1,\ldots,s_{d-1} e_{d-1})$ of $S_{d,k}$. To show the optimality, we use the following auxiliary result:

\begin{theorem}
	\label{tool}
	Let $S$ be a $d$-dimensional simplex in $\R^d$ (not necessarily a lattice simplex) with $k \ge 0$ interior lattice points. Assume that $S$ has a facet $F$ with a unique lattice point $x$ in its relative interior. Let $\beta_1, \ldots, \beta_d > 0 $ be the barycentric coordinates of $x$ with respect to $F$. Then the following hold:
	\begin{enumerate}[(a)]
		\item \label{tool:bound} The volume of $S$ is bounded by 
		\begin{equation}
			\vol(S) \le \frac{k+1}{d! \, \beta_1 \cdots \beta_d}.
		\label{mod:pik:ineq}
		\end{equation}
		\item \label{tool:eqcase} If $k \ge 1$ and \eqref{mod:pik:ineq} is attained with equality, then:
		\begin{enumerate}[1.]
			\item \label{collinear} The point $x$ and the $k$ interior lattice points of $S$ are collinear. That is, there exists a line $g$ that contains $x$ and all interior lattice points of $S$.
			\item \label{edge} The simplex $S$ has an edge parallel to the line $g$.
		\end{enumerate}
	\end{enumerate}
\end{theorem}

Using Theorem~\ref{tool} we obtain

\begin{theorem}[Main result]
	\label{main}
	Let $d \ge 2$ and $k \ge 0$. Then, up to unimodular transformations, $S_{d,k}$ is a unique volume maximizer in $\cS^d(k,1)$.
\end{theorem}

We remark that due to restriction to $\cS^d(k,1)$, it was possible not to exclude the cases $k=0$ and $d=2$. The family $\cS^d(k,1)$ seems to be a natural `neighborhood' of $S_{d,k}$, within which the simplex $S_{d,k}$ is optimal without exceptions. The known exceptional cases for $\cP^d(k)$ are as follows: the volume maximizer in $\cP^2(1)$ is not $S_{2,1}$ but the triangle $\conv(o,3 e_1, 3 e_2)$, while the tetrahedron $S_{3,1}$ is a volume maximizer in $\cP^3(1)$ but not a unique one, the tetrahedron $\conv(o, 2 e_1, 6 e_2, 6 e_3)$ being the other one; see Figure~\ref{fig:dim:3}. If Conjecture~\ref{tool} is true, there are no further exceptions.

\begin{figure}
	\begin{center}
	\includegraphics[height=18mm]{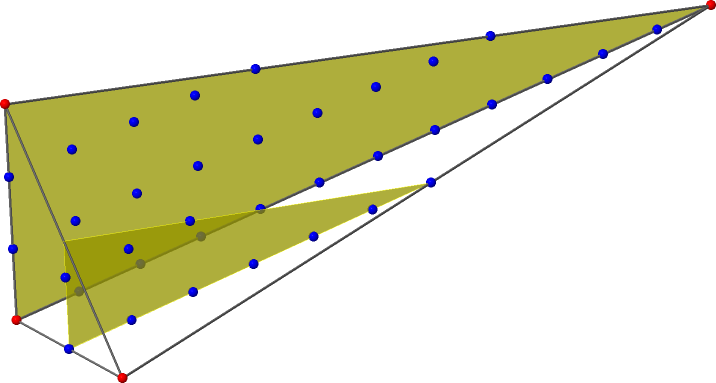}
	\hspace{3em}
	\includegraphics[height=43mm]{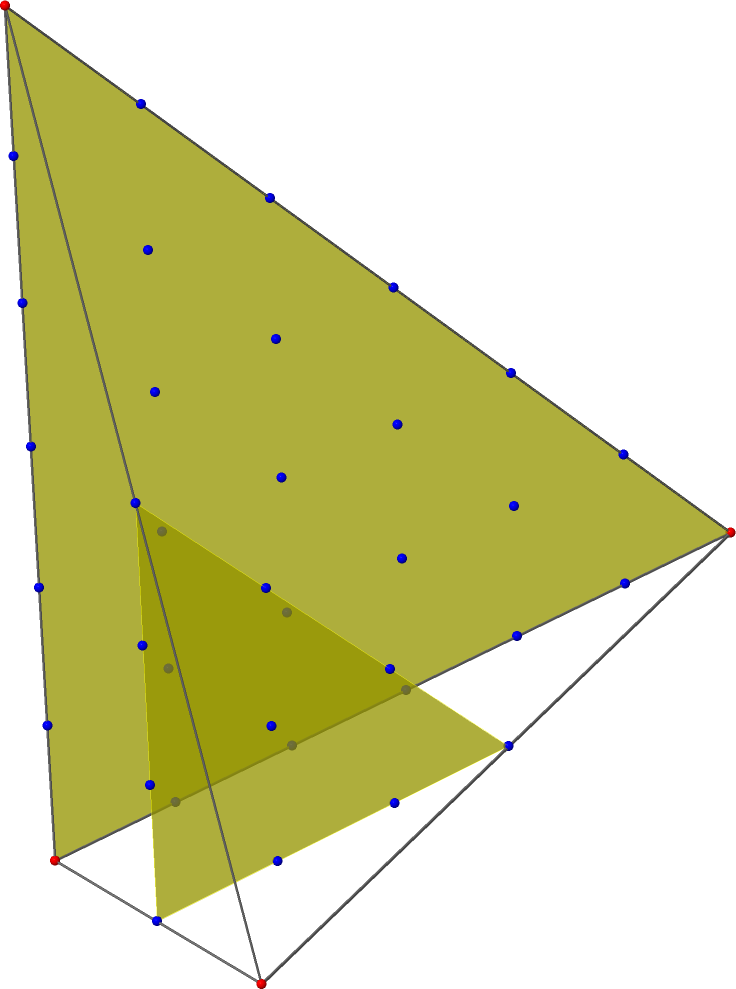}
	\end{center}
	\caption{\label{fig:dim:3} Tetrahedra $S_{3,1} = \conv(o,2 e_1, 3 e_2, 12 e_3)$ and $\conv(o,2 e_1, 6 e_2, 6 e_3)$ are the two volume maximizers in $\cP^3(1)$}
\end{figure}

We give a short outline of our proof approach. Theorem~\ref{tool}\eqref{tool:bound} is a modification of the following result of Pikhurko:

\begin{theorem}[Pikhurko's bound \cite{Pikhurko01}]
	\label{pikhurko}
		Let $S$ be a $d$-dimensional simplex in $\R^d$ (not necessarily a lattice simplex) with $k \ge 1$ interior lattice points. Let $x$ be an interior lattice point of $S$ and let $\beta_1 \ge \ldots \ge \beta_{d+1}> 0$ be the barycentric coordinates of $x$ with respect to $S$, ordered descendingly. Then the volume of $S$ is bounded by 
	\begin{equation}
		\label{pikhurko:bound}
		\vol(S) \le \frac{k}{d! \, \beta_1 \cdots \beta_d}.
	\end{equation}
\end{theorem}

Our proof of Theorem~\ref{tool}\eqref{tool:bound} adapts the proof of Theorem~\ref{pikhurko}. The basic principle in the proofs of  Theorem~\ref{tool}\eqref{tool:bound} and Theorem~\ref{pikhurko} is to link $\vol(S)$ with a volume of an $o$-symmetric compact convex set $B$ and then invoke well-known van der Corput's theorem (see \cite[\S7.2]{GruberLekkerkerker87}), which bounds $\vol(B)$ using the number of interior lattice points of $B$. Theorem~\ref{tool}\eqref{tool:eqcase} is proved using a characterization of the equality case in van der Corput's inequality obtained in \cite{Averkov18}. 

To prove Theorem~\ref{main}, we use Theorem~\ref{tool} and the following result:

\begin{theorem}[On product of barycentric coordinates for $\cS^d(1)$; \cite{AverkovKruemplmannNill15}]
	\label{productBar}
	Let $\tau(S)$ denote the product of the barycentric coordinates of the unique interior lattice point of $S \in \cS^d(1)$. Then the simplex
	\[
	T_d :=\conv(o,s_1 e_1,\ldots,s_d e_d) \in \cS^d(1)
	\]
	is a minimizer of $\tau(S)$ among all simplices $S \in \cS^d(1)$.  The minimizer $T_d$ is unique up to unimodular transformations. The minimum can be expressed as $\tau(T_d)=\frac{1}{(s_{d+1}-1)^2}$.
\end{theorem}

Theorem~\ref{main} consists of an optimality and a uniqueness assertion. The optimality is derived as a straightforward consequence of Theorem~\ref{tool}\eqref{tool:bound} and the optimality part of  Theorem~\ref{productBar}. Once the optimality is established, the uniqueness assertion is derived using  Theorem~\ref{tool}\eqref{tool:eqcase} and the uniqueness assertion of Theorem~\ref{productBar}.

\section{Preliminaries}

Consider $X \subseteq \R^d$. Let $\aff(X),$ $\lin(X),$ $\conv(X)$ and $\intr{X}$ denote the affine hull, linear hull, convex hull and the interior of $X$, respectively. We introduce relative interior $\relintr{X}$ to be the interior with respect to the affine hull of $X$ as the ambient space. The set $X$ is said to be \emph{$o$-symmetric} if $x \in X$ implies $-x \in X$. For $d \ge 2$, we introduce the following notation for the horizontal slice of $X$ at height $t \in \R$:
\[
	\slice_t(X):= \setcond{y \in \R^{d-1}}{(y,t) \in X}.
\]
The cardinality of a finite set $X$ is denoted by $|X|$. 

A set of the form $\Lambda:=\setcond{M z}{z \in \Z^d}$, where $M \in \R^{d \times d}$ is a non-singular matrix, is called a \emph{lattice of rank $d$}, while the value $\det(\Lambda):= |\det(M)|$ is called the \emph{determinant} of $\Lambda$; see \cite{GruberLekkerkerker87}. 

See \cite{Ziegler1995,Gruber2007} for standard terminology and notation from the theory of convex polytopes. The convex hull $S=\conv(v_0,\ldots,v_m)$ of $m+1$ affinely independent points $v_0,\ldots,v_m \in \R^d$ is called an \emph{$m$-dimensional simplex}. Each point $x \in \aff(S)$, can be written uniquely as the affine combination $x = \beta_0 v_0 + \cdots + \beta_m v_m$ with $\beta_0 + \cdots + \beta_m=1$. The values $\beta_0,\ldots,\beta_m$ are called the \emph{barycentric coordinates} of $x$ with respect to the simplex $S$. One has $x \in S$ if and only if all barycentric coordinates of $x$ are nonnegative, and $x \in \relintr{S}$ if and only if all barycentric coordinates of $x$ are strictly positive. 

Apart from Theorem~\ref{productBar}, we will use the following two results:

\begin{theorem}[Van der Corput's inequality; {\cite[\S7.2]{GruberLekkerkerker87}}] 
	\label{vdc}
	Let $\Lambda$ be a lattice of rank $d$ in $\R^d$ and $C \subseteq \R^d$ be an $o$-symmetric compact convex set with non-empty interior. Then 
	\begin{equation}
		\label{vdc:ineq}
		\vol(C) \le \bigl(|\Z^d \cap \intr{C} |+1 \bigr) 2^{d-1} \det(\Lambda).
	\end{equation}
\end{theorem}

In \cite{Averkov18}, an explicit characterization of the equality case in \eqref{vdc:ineq}. This characterization readily implies the following:

\begin{theorem}[On equality case in van der Corput's inequality; see \cite{Averkov18}]
	\label{vdc:eqcase}
	In the notation of Theorem~\ref{vdc} the following holds. If $|\intr{C} \cap \Z^d|> 1$ and \eqref{vdc:ineq} is attained with equality, then $g:=\lin(\Lambda \cap \intr{C})$ is a line and $C$ is a polytope that has an edge parallel to $g$.
\end{theorem}

\section{Proofs}

\begin{proof}[Proof of Theorem~\ref{tool}]
	Let $v_1,\ldots,v_d$ be vertices of $F$ with $x=\beta_1 v_1 + \cdots + \beta_d v_d$.
	Let $v_0$ be the vertex of $S$ not belonging to $F$. Replacing $S$ by $S-v_0$, we assume $v_0=o$. 
	
	\eqref{tool:bound}:
	The linear map $\phi : \R^d \to \R^d$ given uniquely by $\phi(v_1) = e_1,\ldots, \phi(v_d)=e_d$ satisfies
	\begin{align*}
		\phi(S) & =\Delta_0  := \conv(o,e_1,\ldots,e_d),
		\\ \phi(F) & = \Delta  := \conv(e_1,\ldots,e_d).
	\end{align*} 
	The image $\Lambda := \phi(\Z^d)$ of $\Z^d$ under this linear map is a lattice of rank $d$. Clearly, the simplex $\Delta_0$ is related in the same way to the lattice $\Lambda$ as the simplex $S$ to the lattice $\Z^d$. In particular, the vertices $o,e_1,\ldots,e_d$ of $\Delta$ belong to $\Lambda$, the set $\Lambda \cap \intr{\Delta_0} = \phi(\Z^d \cap \intr{S})$ consists of $k$ points, while the set $\Lambda \cap \relintr{\Delta} = \phi(\Z^d \cap \relintr{F})$ consists of exactly one point 
	\[
		b:=\phi(x)=(\beta_1,\ldots,\beta_d).
	\] 
	Consider the $o$-symmetric box 
	\[
		B := [-\beta_1,\beta_1] \times \cdots \times [-\beta_d,\beta_d]
	\]
	and the set $Y := \Lambda \cap \intr{B}$, which can be described as  
	\begin{equation}
		\label{Ydescr}
		Y  = \setcond{(y_1,\ldots,y_d) \in \Lambda}{ |y_1|<\beta_1,\ldots, |y_d|<\beta_d}.
	\end{equation}
	Van der Corput's inequality (Theorem~\ref{vdc}), applied to $B$ and $\Lambda$, yields
	\begin{equation}
		\label{Blambda}
		2^d \beta_1 \cdots \beta_d = \vol(B) \le (|Y| +1) 2^{d-1} \det(\Lambda).
	\end{equation}
	
	Since $\Lambda = \phi(\Z^d)$, the determinant of $\Lambda$ is the ratio by which the volume is changed by the linear map $\phi$:
	\[
		\det(\Lambda) = \frac{\vol(\phi(S))}{\vol(S)} = \frac{\vol(\Delta_0)}{\vol(S)} = \frac{1}{d! \vol(S)}.
	\]
	Thus, \eqref{Blambda} can be reformulated as 
	\begin{equation}
		\label{eq:S:Y}
		\vol(S) \le \frac{|Y|+1}{2 \cdot d! \cdot \beta_1 \cdots \beta_d}.
	\end{equation}
	To show assertion \eqref{tool:bound}, we verify
	\begin{equation}
		\label{Ybound}
		|Y| \le 2 k+1.
	\end{equation}
	The space $\R^d$ can be decomposed into disjoint union of open half-spaces $H^+, H^-$ and a hyperplane $H$ given by:
	\begin{align*}
		H^+ &:= \setcond{(y_1,\ldots,y_d) \in \R^d}{y_1 + \cdots + y_d > 0},
		\\ H^- & := \setcond{(y_1,\ldots,y_d) \in \R^d}{y_1 + \cdots + y_d < 0},
		\\ H &:= \setcond{(y_1,\ldots,y_d) \in \R^d}{y_1 + \cdots + y_d = 0}.
	\end{align*}
	In view of $o$-symmetry of $Y$, one has 
	\[
		|Y| = 2 | H^- \cap Y| + |H \cap Y|.
	\]
	Taking into account \eqref{Ydescr} and the fact that the sum of the components of $b$ is equal to one, we deduce
	\begin{align}
		(H^- \cap Y) + b & \subseteq \Lambda \cap \intr{\Delta_0}, \label{incl:below}
		\\ (H \cap Y) + b & \subseteq \Lambda \cap \relintr{\Delta}. \label{incl:on}
	\end{align}
	The left-hand side of \eqref{incl:on} contains $b$, while the right hand side of \eqref{incl:on} coincides with $\{b\}$. Thus, both sides of \eqref{incl:on} coincide with $\{b\}$. This shows $|H \cap Y| = 1$. Inclusion \eqref{incl:below} implies 
	\begin{align}
		|H^- \cap Y| &\le k. \label{Ybelow:bound}
	\end{align}
	We thus arrive at \eqref{Ybound}. Inequalities~\eqref{eq:S:Y} and \eqref{Ybound} imply assertion \eqref{tool:bound}.
		
	\eqref{tool:eqcase}: Assume that $k \ge 1$ and \eqref{Blambda} is attained with equality. The above arguments imply that \eqref{Blambda}--\eqref{Ybelow:bound} are all satisfied with equality. In particular, 
	\begin{align}
		(H^- \cap Y) + b & = \Lambda \cap \intr{\Delta_0}, \label{H-eq}
		\\ (H \cap Y) + b & = \Lambda \cap \relintr{\Delta} = \{b\}. \label{Heq}
	\end{align}
	
	Since \eqref{Blambda} is attained with equality, Theorem~\ref{vdc:eqcase} implies that $\lin(Y)$ is a line parallel to one of the edges of $B$. This means $\lin(Y) = \lin(e_i)$ for some $i \in \{1,\ldots,d\}$. Then, by \eqref{H-eq} and \eqref{Heq}, $\lin(Y)+b$ is a line that contains $b$ and all points of $\Lambda \cap \intr{\Lambda}$. It follows that the line $g:=\phi^{-1}(\lin(Y)+b)$ contains $x$ and all points of $\Z^d \cap \intr{S}$. Since, $\lin(Y)$ is parallel to the edge $\conv(o,e_i)$ of $\Delta_0$, the line $g$ is parallel to the edge $\conv(o,v_i)$ of $S$. This yields assertion \eqref{tool:eqcase}.
\end{proof}

\begin{proof}[Proof of Theorem~\ref{main}]	
	Assume $d \ge 2$, as otherwise the assertion is trivial.
	Consider an arbitrary $S \in \cS^d(k)$. 
	Let $F$ be a facet of $S$ with $\Z^d \cap \relintr{F}=\{x\}$ and let $\beta_1, \ldots,\beta_d > 0$ be the barycentric coordinates of $x$ with respect to $F$. For the volume of $S$, we obtain the bound
	\begin{align*}
		\vol(S) & \le \frac{k+1}{d! \, \beta_1 \cdots \beta_d} & & \text{(by Theorem~\ref{tool})} 
		\\ & \le \frac{1}{d!} \, (k+1) (s_d-1)^2 & & \text{(by Theorem~\ref{productBar})}
		\\ & = \vol(S_{d,k}).
	\end{align*}
	This verifies the optimality of $S_{d,k}$.

	To show the uniqueness assertion, we assume $\vol(S) = \vol(S_{d,k})$. This implies 
	\[
		\vol(S) = \frac{k+1}{d! \beta_1 \cdots \beta_d} = \frac{1}{d!} (k+1) (s_d-1)^2 = \vol(S_{d,k}).
	\]
	Consequently, $\beta_1 \cdots \beta_d = \frac{1}{(s_d-1)^2}$ and, by characterization of the equality case in Theorem~\ref{productBar} applied for dimension $d-1$, we see that the facet $F$ of $S$ coincides with $T_{d-1} \times \{0\}$, up to unimodular transformations. Changing coordinates using an affine unimodular transformation, we assume 
	\[
		F=T_{d-1} \times \{0\}. 
	\]

	The $(d-1)$-dimensional simplex $T_{d-1}$ has the inequality description
	\begin{equation}
	\label{Td:ineq}
		T_{d-1} = \setcond{(y_1,\ldots,y_{d-1}) \in \R^d}{y_1 \ge 0,\ldots, y_{d-1} \ge 0, \ \frac{y_1}{s_1} + \cdots + \frac{y_{d-1}}{s_{d-1}}  \le 1}.
	\end{equation} 

This representation allows to determine the largest box of the form $[0,\lambda]^{d-1}$ contained in $T_{d-1}$. For the largest box, the vertex $(\lambda,\ldots,\lambda)$ of $[0,\lambda]^{d-1}$ is in the facet $\conv(s_1 e_1,\ldots,s_{d-1} e_{d-1})$ of  $T_{d-1}$, which means that
\[
	\frac{\lambda}{s_1} + \cdots + \frac{\lambda}{s_{d-1}} = 1.
\]
The left-hand side of the latter can be simplified using the well-known equality
\begin{equation}
\label{egyptian}
\frac{1}{s_1} + \cdots + \frac{1}{s_{d-1}} =  1 - \frac{1}{s_d -1}
\end{equation}
for the elements of the Sylvester sequence. Consequently, we obtain $\lambda = \frac{s_d-1}{s_d-2}$. We have thus shown the inclusion
\begin{equation}
\label{cube:in:Td}
\frac{s_d-1}{s_d -2} \ [0,1]^{d-1} \subseteq T_{d-1}.
\end{equation}	
	
	Let $v = (p,h) \in \R^{d-1} \times \R$ be the vertex of $S$ lying outside $F$. Without loss of generality let $h> 0$. The volume of the simplex $S$ can be expressed using its height $h$ and the $(d-1)$-dimensional volume of the respective base as $\vol(S) = \frac{1}{d} h \vol(T_{d-1})$. Since $\vol(T_{d-1}) = \frac{1}{(d-1)!} (s_d-1)$ and $\vol(S) = \frac{1}{d!} (k+1) (s_d-1)^2$, we can determine the height:
	\[
		h =  (k+1) (s_d-1).
	\] 
	For every $t \in [0,h]$,  $\slice_t(S)$ is a homothetical copy of $T_{d-1}$ expressed as 
	\[
		\slice_t(S) = \frac{t}{h} p +(1-\frac{t}{h}) T_{d-1}.
	\]
	In view of \eqref{cube:in:Td}, we arrive at the inclusion
	\[
	\slice_t(S) \supseteq C_t:=\frac{t}{h} p + \Bigl(1- \frac{t}{h} \Bigr) \frac{s_d-1}{s_d-2} [0,1]^{d-1}
	\]
	for $\slice_t(S)$ and the cube $C_t$ with edge length $(1- \frac{t}{h}) \frac{s_d-1}{s_d-2}$. We distinguish the following two cases:
	
	\emph{\underline{Case~1:} $k=0$.} In this case, $C_1 := \frac{1}{h} p + [0,1]^{d-1}$ is a unit cube.  Among the points of the `half-open cube' $\frac{1}{h} p + (0,1]^d$, only the vertex $\frac{1}{h} p + (1,\ldots,1)$ of $C_1$ is in the boundary of $\slice_t(S)$, all the other points being in $\intr{\slice_t(S)}$. This implies that $\frac{1}{h} p + (1,\ldots,1) \in \Z^{d-1}$, as otherwise the unique lattice point in $\frac{1}{h} p + (0,1]^d$ would be in the interior of $\slice_1(S)$, which would contradict $\Z^d \cap \intr{S} = \emptyset$. It follows that the point $q:=\frac{1}{h} p =\frac{p}{s_d-1} $ belongs to $\Z^{d-1}$. Thus, applying the linear unimodular transformation of $\R^d$ that keeps $\R^{d-1} \times \{0\}$ unchanged, and sends $(q,1)$ to $(o,1)$, we can assume  $q=o$. This implies $v=(o,h)= (s_d-1) e_d$ and yields $S = S_{d,0}$. 
	
	\emph{\underline{Case~2:} $k \ge 1$.} The edge length $(1- \frac{t}{h}) \frac{s_d-1}{s_d-2}$ of the cube $C_t$ is strictly larger than one for every $t \in \{1,\ldots,k\}$. It follows that each $C_t$ contains an interior lattice point $p_t$, for every with $t \in \{1,\ldots,k\}$. Consequently, 
	\begin{equation}
		\label{int:lat:points:descr}
		\Z^d \cap \intr{S} = \{(p_1,1),\ldots,(p_k,k)\}.
	\end{equation}
	By Theorem~\ref{tool}\eqref{tool:eqcase}, the unique point $(1,\ldots,1,0)$ in $\Z^d \cap \relintr{F}$ and the points of $\Z^d \cap \intr{S} $ are all collinear. Applying the linear unimodular transformation that keeps $\R^{d-1} \times \{0\}$ unchanged and sends $(p_1,1)$ onto $(1,\ldots,1)$, we assume $p_1 = (1,\ldots,1)$. Under this assumption, the line containing $(1,\ldots,1,0),(p_1,1),\ldots,(p_k,k)$ is parallel to $e_d$. By Theorem~\ref{tool}\eqref{tool:eqcase}, $S$ has an edge parallel to $e_d$. Since the facet $F$ of $S$ is orthogonal to $e_d$, the latter edge connects the vertex $v=(p,h) \not\in F$ of $S$ with one of the vertices $o, s_1 e_1,\ldots, s_{d-1} e_{d-1}$ of $F$. Hence, one has either $p= o$ or $p=s_i e_i$ for some $i \in \{1,\ldots,d-1\}$. For $p=o$, one has $S=S_{d,k}$. If $p=s_i e_i$ and $d=2$, then $S$ coincides with $S_{d,k}$ up to a unimodular transformation. It remains to consider the case $d \ge 3$ and $p=s_i e_i$. In this case, we arrive at a contradiction to \eqref{int:lat:points:descr} by showing 
	\begin{equation}
		\label{contradiction}
		e_1 + \cdots + e_{d-1} \in \intr{\slice_{k+1}(S)}.
	\end{equation}
	In view of $p=s_i e_i$, the set $\slice_{k+1}(S)$ has the following description:
	\begin{align*}
		\slice_{k+1}(S) & = \frac{k+1}{h} (s_i e_i) + \left( 1 - \frac{k+1}{h} \right) T_{d-1}
		\\ & = \frac{s_i}{s_d-1} e_i + \left( 1 - \frac{1}{s_d-1} \right) T_{d-1}
		\\ & = \frac{s_i e_i + (s_d-2) T_{d-1}}{s_d-1}.
	\end{align*}
	This allows to reformulate \eqref{contradiction} as 
	\begin{equation}
		\label{simpleCheck}
		\frac{1}{s_d-2} \bigl ( (s_d-1) (e_1 + \cdots + e_{d-1}) - s_i e_i \bigr) \in \intr{T_{d-1}}
	\end{equation}
	By \eqref{Td:ineq}, in order to show~\eqref{simpleCheck}, it suffices to check the strict inequalities $y_1 > 0,\ldots, y_{d-1} > 0$ and $\frac{y_1}{s_1} + \cdots \frac{y_{d-1}}{s_{d-1}}$ for $y=(y_1,\ldots,y_{d-1})$ being the left-hand side of \eqref{simpleCheck}. This can be done in a straightforward  manner, taking into account \eqref{egyptian} and $d \ge 3$.
\end{proof}

\section{Outlook}

\label{outlook}

\begin{enumerate}[1)]
	\item While Pikhurko's Theorem~\ref{pikhurko} was successfully used to determine the maximum volume in $\cS^d(k)$ for $k=1$ in \cite{AverkovKruemplmannNill15}, determination of the maximum volume in $\cS^d(k)$ for $k \ge 2$ via Theorem~\ref{pikhurko} is doomed to failure for the following reason. For an arbitrary simplex $S \in \cS^d(k)$ and a poitn $x \in \Z^d \cap \intr{S}$, the volume bound on $S$ that we obtain by invoking Theorem~\ref{pikhurko} for $S$ and $x$ is 
\[
\vol(S) \le \frac{k}{d! \beta_1(S,x) \cdots \beta_d(S,x)},
\]
where $\beta_1(S,x) \ge \ldots \ge \beta_{d+1}(S,x) > 0$ are the barycentric coordinates of $x$ with respect to $S$, sorted in the descending order. Thus, with the best choice of $x$, we get the bound 
\[
	\vol(S) \le \nu(S),
\]
where
\[
	\nu(S):= \frac{k}{d!} \ \min_{x \in \Z^d \cap \intr{S}} \ \frac{1}{\beta_1(S,x) \cdots \beta_d(S,x)}.
\]
Balletti and Kasprzyk \cite{arXiv:BallettiKasrpzyk16} enumerated the family $\cS^3(2)$, up to unimodular transformations. Their enumeration allows to check that, for $59$ out of $471$ tetrahedra $S \in \cS^3(2)$, the strict inequality $\nu(S) > \vol(S_{3,2})$ holds. Pikhurko's bound from Theorem~\ref{pikhurko} is too weak for these tetrahedra. Thus, there is a need in new approaches to bounding the volume of simplices in $\cS^d(k)$. Theorem~\ref{main} is a first step in this direction. 
	\item It would be interesting to compare the cardinality of $\cS^d(k,1)$ and $\cS^d(k)$ for arbitrary $d \ge 2$ and $k \ge 1$ (with respect to identification of lattice simplices up to unimodular transformations). 
	For large values of $d$, the cardinality of $\cS^d(k,1)$ must be large, since every simplex $T \in \cS^{d-1}(1)$ with $o \in \intr{T}$ gives rise to the simplex $S = \conv(T \times \{0\} \cup \{(k+1) e_d\}) \in \cS^d(k,1)$, where the cardinality of $\cS^{d-1}(1)$ is large. Using the database of Balletti and Kasprzyk, we verified that $183$ out of $471$ tetrahedra in $\cS^3(2)$ belong to $\cS^3(2,1)$. 
	\item We formulate a natural counterpart of Conjecture~\ref{conj} in the case $k=0$. Consider the family $\cP_{\max}^d(0)$ of all lattice polytopes with no interior lattice points that are maximal within $\cP^d(0)$ with respect to inclusion.  The family $\cP_{\max}^d(0)$ occurs in integer optimization and algebraic geometry; see \cite{AverkovWagnerWeismantel11} and \cite{BHHSarxiv2016}, respectively. 
	\begin{conjecture}
	\label{conj:k=0}
	Up to unimodular transformations, the simplex $S_{d,0}$ is a unique volume maximizer in $\cP_{\max}^d(0)$.
	\end{conjecture}

	Our Theorem~\ref{main} provides support for the positive answer. Conjecture~\ref{conj:k=0} is true in dimension two for trivial reasons, as the triangle $S_{2,0}$ is the unique element of $\cP_{\max}^2(0)$, up to unimodular transformations. Conjecture~\ref{conj:k=0} is also true in dimension three, which follow from the complete enumeration of $\cP^3_{\max}(0)$ established in  \cite{AverkovWagnerWeismantel11,akw2017notions}.
\end{enumerate}

\subsection*{Acknowledgements}

I would like to thank Alexander Kasprzyk and Gabriele Balletti for sharing their database of $\cP^3(2)$. 

\bibliography{lit2}
\bibliographystyle{amsalpha}

\end{document}